\documentclass[12pt, leqno]{amsart}
\usepackage{amscd,amssymb,amsmath,amsthm}

\theoremstyle{plain}
\newtheorem{theorem}{Theorem}[section]
\newtheorem{Lem}[theorem]{Lemma}

\newtheorem*{theorem*}{Theorem}
\newtheorem{Cor}[theorem]{Corollary}

\newcommand{\diff}{\textrm{Diff}^{\: 1}(M)}
\newcommand{\ep}{\epsilon}

\newcommand{\rem}{\medskip \noindent \textbf{Remark: }}
\newcommand{\real}{\mathbb{R}}
\newcommand{\bb}{\mathbb}
\newcommand{\D}{\displaystyle}
\newcommand{\bbR}{\mathbb{R}}

\author{Anne E. McCarthy}
\title{rigidity of trivial actions of abelian-by-cyclic groups }

\begin{document}
\maketitle
{\footnotesize 
  \centerline{ Department of Mathematics }
  \centerline{ Fort Lewis College}
  \centerline{ Durango, CO, 81301} }

\begin{abstract} Let $\Gamma_A$ denote the abelian-by-cyclic group associated to an integer-valued, non-singular matrix $A$.  We show that if $A$ has no eigenvalues of modulus one, then there are no faithful $C^1$ perturbations of the trivial action $ \iota: \Gamma_A \to \diff$, where $M$ is a compact manifold.
\end{abstract}
\medskip
\section{introduction}

The question of existence and stability of global fixed points for group actions has been studied in many different contexts. In the setting of actions of Lie groups it was shown by Lima \cite{lima} that $n$ commuting vector fields on a genus $g$ surface, $\Sigma_g,$ of non-zero Euler characteristic have a common singularity.  This implies that any action of the abelian Lie group $\real^n$ on $\Sigma_g$ has a global fixed point.  It was later shown by Plante \cite{plante}  that any action of a nilpotent Lie group on a surface with non-zero Euler characteristic has a global fixed point.  

The study of stability of global fixed points for group actions is also related to the study of foliations. Given a foliation of a manifold with a compact leaf $L$ and a transverse disk $D,$ the holonomy map along $L$ defines an action of $\pi_1(L)$ on the disk $D$.  Perturbations of group actions are related to the study of foliations, for given a nearby leaf $L'$ diffeomorphic to $L$, the holonomy along $L'$ defines a new action that is a perturbation of the original.  The Thurston stability theorem gives conditions for local stability of $C^1$ foliations of a compact manifold.  Methods of Thurston were then modified by Langevin and Rosenberg \cite{langv}, Stowe \cite{stow1} \cite{stow2}, and Schweitzer \cite{schw} to establish results regarding stability of global fixed points, and leaves of fibrations.

Inspired by the ideas of Lima \cite{lima}, Bonatti \cite{bon} used methods similar to Thurston's to show that any $\mathbb{Z}^n$ action on surfaces with non-zero Euler characteristic generated by diffeomorphisms $C^1$ close to the identity has a global fixed point. Using similar techniques, Druck, Fang and Firmo \cite{dff} proved a discrete version of Plante's theorem.  The dynamics of group actions generated by real analytic diffeomorphisms close to the identity were studied by  Ghys \cite{ghys}, who showed that such actions were either recurrent, or displayed a property similar to solvability of the group.

We examine how a particular family of solvable groups acts via diffeomorphisms close to the identity on a compact manifold.  We show that such actions display near-rigidity, in the sense that there are no faithful actions with generators close to the identity.  It is an easy consequence that actions close to the identity of these groups on a compact surface of non-zero Euler characteristic have a global fixed point. 

Let $A$ be an invertible $n \times n$ matrix with integer entries.  To the matrix $A$ one can associate the solvable group
$$\Gamma_A = \langle a, b_1, ... b_n \  | \ b_ib_j=b_jb_i, \ ab_ia^{-1} =\prod_j b_j^{A_{ij}} \rangle. $$
For instance, when $ A= \left(\begin{array}{cc} 2 & 3 \\ 4 & 5 \end{array}\right) $ the associated group is 
$$\Gamma_A = \langle a, b_1, b_2 \ | \ b_1b_2=b_2b_1, \  ab_1a^{-1} = b_1^2b_2^3,  \  ab_2a^{-1} = b_1^4b_2^5 \rangle. $$
In the case where $A$ is a $1 \times 1$ matrix, $A = [n] $, the associated group is the \textit{Baumslag-Solitar group},
$$BS(1,n) = \langle a, b \ | \ aba^{-1} = b^n \rangle. $$
These groups have a geometric interpretation as the fundamental group of the space $\mathbb{T}^n \times [0,1]/ \sim $ where ends are glued via the toral endomorphism induced by $A$.     A group $\Gamma$ is said to be \textit{abelian-by-cyclic} if there exists an exact sequence
$$ 1 \to \mathcal{A} \to \Gamma \to Z \to 1,$$
where the group $\mathcal{A}$ is abelian, and $Z$ is an infinite cyclic group.  Note that the commutator subgroup $[\Gamma, \Gamma]$ is contained in $ \mathcal{A}$, so all such $\Gamma$ are solvable groups.  The class of all finitely presented, torsion free, abelian-by-cyclic groups is exactly given by groups of the form $\Gamma_A$. See \cite{fm} for a nice proof of this.

We are interested in actions of the groups $\Gamma_A$ on compact manifolds in the case where the matrix $A$ does not have an eigenvalue of modulus one. Given a finitely generated group $\Gamma$ and a manifold $M$, a \textit{$C^r$ action of $\Gamma$ on $M$} is a homomorphism $\rho: \Gamma \to \textrm{Diff}^{\, r}(M)$. We commonly refer to this homomorphism as a representation (into $\textrm{Diff}^{\, r}(M)$) and use the associated language.  The representation $\rho$ is said to be \textit{faithful} if $\rho$ is injective.  We denote by $\mathcal{R}^r(\Gamma, M)$ the collection of all representations of $\Gamma$ into $\textrm{Diff}^{\, r}(M)$.  

We now fix some notations that we use for the remainder of this exposition.  Let $M$ be a compact manifold embedded in $\mathbb{R}^\ell$.  Let $ \| \cdot \| $ denote the standard Euclidean metric on $\mathbb{R}^\ell$. Given a $C^1$ diffeomorphism $h$, we will define
$$ \| h \|_1 = \sup_{x \in M} \{ \| h(x)\| + \| D_x h \| \}.$$
With this we define the $C^1$ distance
$$ d(h, g) = \| h-g \|_1. $$

The collection of $C^1$ representations $\mathcal{R}^1(\Gamma, M)$ carries a topology.  For the group $\Gamma$, fix a generating set $ \langle \gamma_1, ..., \gamma_k \rangle$, and let
$$d_{_{C^1}}(\rho_{_1}, \rho_{_2}) = \sup_{\gamma_1, ... , \gamma_k} d_{_{C^1}}(\rho_{_1}(\gamma_i), \rho_{_2}(\gamma_i)). $$
This distance depends on the choice of generating set for the group $\Gamma$.   However, given any two generating sets the associated metrics are equivalent. We now state our main theorem.

\begin{theorem}
\label{thm.abcnoact} Let $M$ be a compact manifold and $A$ be a non-singular $n \times n$ matrix with integer entries. Let $\langle a, b_1, \dots b_n \rangle$ be generators $\Gamma_A$.  If $A$ has no eigenvalue of modulus one, then there exists $\ep >0$ such that any $C^1$ action $\rho: \Gamma_A \to \diff$  with $d_{C^1}(\rho,id) < \ep$ is not faithful.  In particular, $ \rho(b_i) = id $ for all $i = 1,2, \cdots n$.
\end{theorem}

\rem This can be viewed as a rigidity result in the following sense.  Consider the trivial representation 
$\iota: \Gamma_A \to \textrm{Diff}^{\,1}(M),$ given by $\iota(a) = \iota(b_i) =id$.  Note that any assignment $\rho: \Gamma_A \to  \textrm{Diff}^{\,1}(M)$ with  $\rho(a) = f$ and $\rho(b_i) =g_i =id$ determines an action.  Therefore, there are infinitely many perturbations of the trivial representation of the form $\rho(g_i) = id$. The main theorem states that representations with $\rho(b_i) =id$ are in fact the only $C^1$ perturbations of the trivial action.  The compactness hypothesis is necessary for this theorem, see remarks at the end of Section 3 for further discussion.

This result is not true for $C^0$ perturbations. The groups $\Gamma_A$ are discrete subgroups of the affine group. The standard representations are given by $\alpha(a)(x) = \lambda x$, and $\alpha(b_i)(x) = x+v_i$, where $\lambda$ is an eigenvalue of $A$ with corresponding eigenvector $v$. An example of Lima \cite{lima} gives two vector fields for which the associated flows give an action of the affine group on the sphere $S^2$.  This example produces actions arbitrarily close to the identity in $\mathcal{R}^0(\Gamma_A, S^2)$, by taking sufficiently small time-$t$ maps of the flows by which the affine group acts. In fact, this action has no global fixed point.

Under the hypotheses of Theorem \ref{thm.abcnoact}, we characterize global fixed points:
\begin{Cor} If $d_{C^1}(\rho - id) < \ep,$ then any fixed point of $\rho(a)$ is a global fixed point for $\rho.$
\end{Cor}

\noindent
\emph{Acknowledgments.}
This project grew out of dissertation work completed at Northwestern University.  Special thanks are extended to Amie Wilkinson for her guidance and many helpful conversations during the course of this project.  I would also like to thank Christian Bonatti for discussions conveying valuable intuitions regarding the behavior of diffeomorphisms close to the identity.  
\section{Preliminary Tools}

Central to Thurston's argument is that given a non-trivial holonomy map $H$ about a compact leaf $L$, if the linear action $dH : \pi_1(L) \to GL(k, \bbR^k)$ is trivial then there is non-trivial representation $u: \pi_1(L) \to \bbR^k$.  Stowe  \cite{stow1} later showed that in the case where the linear action is non-trivial, there is an analogous cocyle $u: G \to \bbR^k$. Further results regarding how this representation behaves under perturbation have been proved by Langevin and Rosenberg \cite{langv}, Stowe  \cite{stow1}, \cite{stow2}, and Schweitzer \cite{schw}.  This body of work establishes the following lemma.

\begin{Lem}[Thurston]
\label{lem.TrivOrHomo} Let $G$ be a group with finitely many generator $g_1, g_2, \dots g_r$, and let $M$ be a manifold embedded in $\bbR^\ell$.  Suppose there exist an action $\rho : G \to \diff$ and a point $p \in M$ for which $\rho(g)(p)=p$ for all $g \in G$, and $D_p\rho_0(g) = id$ for all $g \in G$.  Then either 
\begin{enumerate}
\item There is a neighborhood of $p$ consisting of points that are fixed by all actions near $\rho$ in $\mathcal{R}^1(G,M)$ or
\item There exist actions $\rho_k \to \rho$ in $\mathcal{R}^1(G, M)$ and points $x_k \to p$ in $M$ with $x_k$ not fixed by $\rho_k$ for which the sequence $u_k : G \to \bbR^\ell$ given by 
$$ u_k(g) = \frac{\rho_k(g)x_k - x_k}{ \max \left\{ \left\| \rho_k(g_j)x_k - x_k  \right\| : 1 \leq j \leq r \right\} } $$
converges to a non-trivial homomorphism $u: G \to T_pM$.
\end{enumerate}
\end{Lem}

The vector $ f(x) -x$ in $\mathbb{R}^\ell$ indicates how the point $x$ is moved by $f$. The above lemma characterizes such \textit{displacement vectors}. The homomorphism $u$ associates to each $g \in G$ a normalized displacement vector.  For points $x$ near the fixed point $p$, we note that the displacement of $x$ by $\rho(g_1g_2)$ is close to the vector sum of the displacements associated to $\rho(g_1)$ and $\rho(g_2)$.  Bonatti \cite{bon.thesis} modified the ideas of Thurston to provide estimates on how much the assignment $ g \mapsto \rho(g)(x)-x $ can differ from a homomorphism when the image of $\rho$ is close to the identity in $\mathcal{R}^1(G, M)$.  These modifications do not require a global fixed point with trivial linear action.

\begin{Lem}[Bonatti]
\label{lem.bon2}
For all $\eta > 0$, $N \in \mathbb{N}$, there exists $\ep > 0$ such
that the following property holds: Let $f_1, f_2, ... f_N$ be
diffeomorphisms $f_i : M \to M$ such that $d_{C^1}(f_i, id) < \ep$.
Then for all $x \in M,$
$$\Big\| \left( f_1 \circ f_2 \circ \cdots \circ f_N - id \right)(x) -
\sum_{i=1}^{N} \left( f_i -id \right)(x) \Big\| $$ 
$$ < \eta \sup_{i} \| \left(f_i - id \right)  (x)\|.$$
\end{Lem}

As we proceed, we will consider displacements in $\bb{R}^\ell$ of $n$ different functions simultaneously using an $n \times \ell$ displacement matrix.  The following elementary bound concerning entries of a matrix will be useful.  
\begin{Lem} 
\label{lem.matrixbds}
 Let $A$ be an $n \times \ell$ matrix, with row vectors ${a}_i$ and columns $v_j$.  Then for all $1 \leq i \leq n$, 
 $$ \| {a}_i \| \leq \sqrt{\ell}  \sup \{ \| {v}_j  \|, 1\leq j \leq \ell \}. $$
\end{Lem}

\begin{proof}
Note that
$$ \sum_{i=1}^{n} \| {a}_i \|^2 = \sum_{j=1}^{\ell} \| {v}_j \|^2, $$
since both of these quantities give the sum of the square of all the entries of $A$.  Therefore,  for all $1 \leq i \leq n $,  
$$ \| {a}_i \|^2 \leq  \sum_{i=1}^{n} \| {a}_i \|^2 = \sum_{j=1}^{\ell} \| {v}_j \|^2 \leq \ell  \sup \{ \| {v}_j  \|^2, 1\leq j \leq \ell \}. $$
\end{proof}

It will also be useful to rewrite the composition $a^kb_ia^{-k}$ in terms of the generators $b_1, \dots, b_n$. Here $(A^k)_{ij}$ is the entry of $A^k$ in the $(i,j)$ position.

\begin{Lem} 
\label{lem.fk}
For generators $a$ and $b_i$ of $\Gamma_A$,   $$a^{k}b_ia^{-k} = \prod_ja_j^{(A^k)_{ij}}.$$
\end{Lem}

\begin{proof} We use induction on $k$, where the group relation $ab_ia^{-1} = \prod_j b^{A_{ij}} $ establishes the base case. 
Suppose that $a^{k}b_ia^{k} = \prod b_\ell^{(A^k)_{i\ell}}$.
This implies that 
$$ a^{k+1}b_ia^{-(k+1)} = f \left( \prod_\ell b_\ell^{(A^k)_{i\ell}} \right)a^{-1} = \prod_\ell( ab_\ell a^{-1})^{(A^k)_{i\ell}}.$$
Applying the group relation to $ab_\ell a^{-1}$ we get that 
$$ a^{k+1}b_ia^{-(k+1)} = \prod_\ell \left( \prod_j b_j^{A_{\ell j}}\right)^{(A^k)_{i\ell}}  = \prod_j b_j^{\sum_\ell (A^k)_{i \ell}A_{\ell j}}.$$
Which shows that 
$a^{k+1}b_ia^{-(k+1)} = \prod_jb_j^{(A^{k+1})_{ij}}.$ 
\end{proof}

\section{Overview of Proof of Theorem \ref{thm.abcnoact}}

Suppose $\rho: \Gamma_A \to \textrm{Diff}^{\, 1}(M)$ is an action.  Let $\rho(a) = f$ and $\rho(b_i) = g_i$ denote the images of the generators of $\Gamma_A$ within $\textrm{Diff}^{\, r}(M)$.  For the action to be close to the identity means that the functions $f$ and $g_i$, $1 \leq i \leq n$ all have distance to the identity in $\mathcal{R}^1(G,M)$ less than $\ep$.  We apply our understanding of behavior of diffeomorphisms close to the identity to examine how displacements by the functions $g_i$ vary as we move from a point $x$ to $f(x)$.

In particular, we show that $g_i$-displacements of the point $x$ are related to those at the point $f(x)$ by the matrix $A$ that defines the group.  This is seen fairly easily in the case that the matrix $A = [n]$ is a $1 \times 1$ matrix.  Suppose $y = f(x),$ then  
$$\|g(x)-x\| = \|g f^{-1}(y) - f^{-1}(y)\|$$  
$$\approx  \frac{1}{\| D_{g(x)} f \| } \|fgf^{-1}(y) - (y)\|  \approx \|g^n (y) -y\|$$  
$$\approx n \|g(y) -y\|$$ 
The above approximations use that the derivative of $f$ has norm close to one, and that the displacement resulting from a composition can be viewed as a sum of displacements.  More generally, we establish for the functions $g_1, \dots, g_n$ that if we collect all displacements of a point $x \in \mathbb{R}^\ell$ into an $n \times \ell$ matrix $D(x)$, then $D(x) \approx AD(y).$

We then examine the consequences of having the displacement by the functions $g_1, \dots, g_n$ at different points related via the application of a hyperbolic matrix.  On the one hand, by assuming that the action is close to the identity, we are assuming that all the displacements by the functions $g_j$ are uniformly small.  On the other hand, displacements of the point $x$ are related to those of the point $f(x)$ by application of the hyperbolic matrix $A$. In order to conclude the proof, use compactness to choose $x$ to be the point with a sort of maximal displacement. Follow displacements along a partial orbit $x, f(x), \dots, f^k(x)$ until displacements have been expanded by $A$.  These expanded displacements contradict the choice of $x$ unless the original displacements were zero.  Note that the estimates provided in Lemma \ref{lem.bon2} depend on the number of functions that are composed, so we must fix the number of iterates needed to detect expansion.

It is of interest that this argument is very different from those that use Thurston's stability theorem.  The contradiction in this argument occurs at a place of maximal displacement for $\rho(b_i)$, and not a global fixed point.  In fact, one can locally define a faithful action $\rho$ with a global fixed point $p$ for which $D(\rho(g))(p) =id$: The standard affine action $\rho(a)(x) = \lambda x$, $\rho(b_i)(x) = x + v_i$ fixes infinity.  On a positive half-neighborhood of infinity we modify this action using ideas of Navas \cite{nav1}.  Conjugate $F = \phi^{-1} \circ f \circ \phi$ and $G_i = \phi^{-1} \circ g_i \circ \phi$, where $\phi = e^{\frac{1}{x}}$.  The result is an action defined on $[0,1)$ that fixes $0$, with $F'(0)=G_i'(0)=1$.   This can be extended by defining the functions $G_1, \dots, G_n$ to be the identity on $(-1,0)$ and any appropriate choice of $F$ on $(-1,0)$ to produce an action on the interval $(-1,1).$ 

\section{$g_i$-Displacement Along Orbits of $f$}

Before proceeding, we fix a number of constants.   Given an $n \times n$ matrix $A$ with no eigenvalues of modulus one there exists an $A$-invariant splitting
$$\bb{R}^n = \bb{E}^u \oplus \bb{E}^s $$
and constants $k \in \bb{N}$, $\theta^u > 1$, and $0 < \theta^s < 1$ with the following properties:
$$ \| A^k v \| > \theta^u\|v\| \ \ \textrm{for all} \ \ v \in \bb{E}^u \ \ \ \textrm{and} \ \ \ \| A^k v \| < \theta^s \|v\| \ \ \textrm{for all} \ \ v \in \bb{E}^s. $$
We fix $k$ to be the least positive integer that ensures this expansion and contraction, and fix the corresponding $\theta^u$ and $\theta^s$.  We also fix $N \in \bb{N}$ to based on the number of composed functions to which we will apply Lemma \ref{lem.bon2}.  In particular, set
$$ N > \max_i \sum_j \left|A^k_{ij}\right|. $$
  Next, fix $\alpha < 1/2$ and set
$$C_k = \frac{\| A^k \| + \alpha}{1- 2\alpha} $$
where $\| A^k \|$ denotes the operator norm.
Select $0< \eta < \frac{\alpha}{\sqrt{n \ell}}$ small enough to ensure that
$$ 2 \eta \sqrt{n \ell} (2C_k + 1) < \min \left\{ \theta^u - 1, 1 - \theta^s \right\}. $$
We now choose $\eta_1$ such that $d(f, id) < \eta_1$ implies that $d(f^k, id) < \eta$, and $d(f, id) < \eta$.
For a $C^1$ function, we know by definition that $f(x + y) - f(x) = Df(y) + o(y).$  Furthermore we an find an $\ep_0$ such that if $ \|y \| < \ep_0$, then  $o(y) < \eta\|y \|$.
Let $\ep_1$ be selected so that Lemma \ref{lem.bon2} holds for $N$ and $\eta$ as selected above. Finally, set $\ep = \min \{ \eta_1, \ep_1, \ep_0 \}.$  This will be the $\ep$ for which the theorem holds.

We now introduce some notation. All preliminary estimates apply to displacement vectors of the form $h(x)-x,$ where $h$ can be any of the generating diffeomorphisms $f$ or $g_i$, and $h(x)-x$ is a vector in $\mathbb{R}^\ell$.  For this reason, we write $H(x) = h(x)-x.$  We think of $H(x)$ as a vector that will sometimes be written in components as an $\ell$-dimensional row vector.   To each point $x$ in $M$, there is an associated $n \times \ell$ matrix, $D(x),$ of displacements by the functions $g_i$,
$$ D(x) = \left[  \begin{array}{c} G_1(x) \\ \vdots \\ G_n(x) \end{array} \right]. $$
Specifically, the $i$th row of $D$ is the $\ell$-dimensional row vector $g_i(x) -x$, corresponding to the displacement of $x$ by the diffeomorphism $g_i$.

Next, we examine how the displacement matrix $D(x)$ varies as we move along orbits of $f$.  We will see that $D(x) \approx AD(f(x))$.
 
\begin{Lem}
\label{lem.hyp} Let  $ (D(x))_j$ denote the $j$th column of the displacement matrix $D(x)$ at the point $x$.  If $y = f^{k}(x)$, then 
$$ \| (D(x))_j - A^k(D(y))_j \| < \eta \sqrt{n \ell} \left(  2 \sup_j \| (D(x))_j \| +  \sup_j \| (D(y))_j \| \right).$$ for all $ 1 \leq j \leq \ell$.
\end{Lem}

\begin{proof}
Suppose we are given $f \in \diff$ with $d_{C^1}(f^k, id)< \eta$ and let $f^{-k}(y) =x$.  By the definition of the derivate, we know that
$$f^k(g_i(x)) = f^k(x) + D_{x}f^k(G_i(x)) + o(\|G_i(x) \|).$$
From the selection $\ep_0$, we know that
$$o(\|G_i(x)\|) < \eta \|G_i(x)\|, $$
thus allowing us to bound
$$ \| f^kg_i(x) - f^k(x) -D_{x}f^k(G_i(x)) \| < \eta \|G_i(x)\|.$$
Letting $x=f^{-k}(y)$, and using that $\| D_{x}f^k - Id \| < \eta $ this implies that
$$ \| (f^kg_{i}f^{-k}(y) - y) - (G_i(x)) \| < 2 \eta \| G_i(x) \|,$$
which, by Lemma \ref{lem.fk} is equivalent to the bound
$$ \left\| \left( \prod_{j} g_j^{(A^k)_{ij}} - id \right) (y)  - G_i(x)  \right\| < 2 \eta \| G_i(x) \|.$$
By Lemma \ref{lem.bon2}, we know that
$$ \left\| \left( \prod_j g_j^{(A^k)_{ij}} -id \right)(y) - \sum_j (A^k)_{ij}(g_j -id)(y) \right\| < \eta \sup_j \{ \| (g_j-id)(y)\| \}. $$
So that for each $i = 1,2, ...n$, and for all $x \in M,$
$$ \left\|  \left( A^k \left[ \begin{array}{c} G_1(y) \\ \vdots \\ G_n(y)
\end{array} \right]_i    - G_i(x) \right) \right\| < \eta \| G_i(x) \| + \eta \sup_j \|G_j(y) \| .$$

Specifically, this inequality tells us that the $i$th row of $D(x)$ differs very little in norm from the $i$th row of $A^kD(y)$.  This implies that the specific entries $(D(x))_{ij}$ and $(A^kD(y))_{ij}$ must also satisfy
 $$\left| (D(x))_{ij} - (A^kD(y))_{ij} \right| <   \eta \left(2  \sup_i \| G_i(x) \| +  \sup_j \|G_j(y) \| \right).$$
Therefore the columns must satisfy
$$ \| (D(x))_j - A^k(D(y))_j \| < \eta \sqrt{n} \left(  2 \sup_i \| G_i(x) \| +  \sup_j \|G_j(y) \| \right).$$
By application of Lemma \ref{lem.matrixbds}, we can bound the norms $\| G_i(x) \|$ and $\| G_i(y) \|$ using the columns of the corresponding displacement matrices:
$$ \| (D(x))_j - A^k(D(y))_j \| < \eta \sqrt{n \ell} \left(  2 \sup_j \| (D(x))_j \| +  \sup_j \| (D(y))_j \| \right).$$
\end{proof}

In order to prove the theorem we must modify the estimates of Lemma \ref{lem.hyp}, so that our upper bound depends on displacements at only one point.

\begin{Lem}
\label{lem.Ck}  There exists a constant $C_k > 0$ such that for all $x \in M$
$$\sup_i \| (D(x))_i \| < C_k \sup_j \| (D(f^k(x)))_j \| $$
\end{Lem}

\begin{proof} It suffices to show for $\eta < \frac{\alpha}{\sqrt{n\ell}}$.  Let $x \in M$ be given.  Select $1 \leq j_0 \leq \ell$ such that 
$$ \| (D(x))_{j_0} \| = \sup_j \|(D(x))_j \|. $$
Applying Lemma \ref{lem.hyp} to $x$, $j_0$ and $y = f^k(x)$ we see that
$$ \| (D(x))_{j_0} - A^k(D(y))_{j_0} \| < \eta \sqrt{n \ell} \left(  2  \| (D(x))_{j_0} \| +  \sup_j \| (D(y))_j \| \right).$$
Which implies that 
$$(1 - 2 \alpha)  \| (D(x))_{j_0} \| <  \| A^k (D(y))_{j_0} \| + \alpha \sup_j \| (D(y))_j \| .$$
So we can conclude that
$$(1 - 2 \alpha)  \| (D(x))_{j_0} \| < ( \| A^k \| + \alpha ) \sup_j \| (D(y))_j \| .$$
Setting $\D C_k = \frac{\| A^k \| + \alpha}{1 - 2 \alpha} $ we get the desired result.
\end{proof}

\section{Proof of Theorem \ref{thm.abcnoact}}

We now conclude the proof of Theorem \ref{thm.abcnoact}.  
\begin{proof}
Choose $z$ and $j_1$ to be the values that attain the supremum
$$ \sup_{x \in M, 1 \leq j \leq \ell} \left\{ \| \pi_u( (D(x))_j) \|, \| \pi_s ((D(x))_j) \| \right\}. $$
We consider two different cases.  If the supremum is attained by 
$$ \| \pi_u( (D(z))_{j_1}) \| = \sup_{x \in M, 1 \leq j \leq \ell} \left\{ \| \pi_u( (D(x))_j) \|, \| \pi_s( (D(x))_j) \| \right\}, $$
we note that 
$$ \| (D(z))_{j_1} \| \leq \| \pi_u( (D(z))_{j_1}) \| +\| \pi_s((D(z))_{j_1}) \| \leq 2 \| \pi_u( (D(z))_{j_1}) \|. $$
Consider the point $x = f^{-k}(z)$, and apply Lemmas \ref{lem.hyp} and \ref{lem.Ck}.  We see that
$$ \| \pi_u( (D(x))_{j_1}) - \pi_u( A^k (D(z))_{j_1}) \| < 2 \eta \sqrt{n \ell}(2C_k + 1 ) \| \pi_u( (D(z))_{j_1}) \|. $$
We conclude that 
$$ \| \pi_u( (D(x))_{j_1}) \| > ( \theta^u - 2 \eta \sqrt{n \ell}(2C_k + 1 ) ) \| \pi_u( (D(z))_{j_1}) \| > \| \pi_u( (D(z))_{j_1} )\| ,$$
where the last inequality follows from our choice of $\eta$.

Likewise if the supremum is attained by $ \| \pi_s( (D(z))_{j_1}) \| $, we set $y =f^k(x)$, and are able to conclude
$$\| \pi_s( (D(y))_{j_1}) \| > ( \theta^s + 2 \eta \sqrt{n \ell}(2C_k + 1 ) ) \| \pi_s( (D(z))_{j_1}) \| > \| \pi_s( (D(z))_{j_1}) \|.$$
So we conclude $D(x) = 0$ for all $x$, thus proving the theorem.
\end{proof}

\end{document}